\documentclass[12pt]{article}

\usepackage{amssymb}
\usepackage{url}
\usepackage{amscd}
\usepackage{amsthm}
\usepackage{amsmath}
\usepackage{graphicx,tikz}

\newtheorem {Theorem}  {Theorem}

\newtheorem {Definition} {Definition}

\begin{document}
\baselineskip = 15pt
\bibliographystyle{plain}

\title{Friends-and-strangers is \textsf{PSPACE}-complete}
\date{}
\author{Chao Yang\\ 
              School of Mathematics and Statistics\\
              Guangdong University of Foreign Studies, Guangzhou, 510006, China\\
              sokoban2007@163.com, yangchao@gdufs.edu.cn\\
              \\
        Zhujun Zhang\\
              Government Data Management Center of\\
              Fengxian District, Shanghai, 201499, China\\
              zhangzhujun1988@163.com\\
              }

\maketitle

\begin{abstract}
In this paper, we show that the friends-and-strangers problem is \textsf{PSPACE}-complete by reduction from the \textsc{Ncl} (non-deterministic constraint logic) problem.
\end{abstract}


\section{Introduction} 

Recently, the friends-and-strangers problem was introduced by Defant and Kravitz \cite{dk21}. Given two finite simple graphs $X$ and $Y$ with the same order $n$. The vertex set $V(X)$ of graph $X$ is a set of locations, and $x_1x_2\in E(X)$ means the two locations $x_1$ and $x_2$ are adjacent. The vertex set $V(Y)$ of graph $Y$ is a set of people, and $y_1y_2\in E(Y)$ means two people $y_1$ and $y_2$ are friends while $y_1y_2\not\in E(Y)$ means they are strangers. A bijection $\sigma: X \rightarrow Y$ is called a \textit{configuration}, which is an arrangement of the people in $V(Y)$ on locations in $V(X)$. A new configuration $\sigma'$ can be obtained from a configuration $\sigma$ by a \textit{swap} if and only the two bijections $\sigma'$ and $\sigma$ differ at exactly two points $x_1, x_2 \in X$, where $x_1$ and $x_2$ are adjacent in $X$, and $\sigma(x_1)$ and $\sigma(x_2)$ are adjacent in $Y$. In other words, $\sigma'(x_1)=\sigma(x_2)$, $\sigma'(x_2)=\sigma(x_1)$, and $\sigma'(x)=\sigma(x)$ for all $x\in X - \{x_1, x_2\}$. Intuitively speaking, two people can swap their locations if and only if the two locations are adjacent and the two people are friends. The \textit{friends-and-strangers graph}, denoted by $\textsf{FS}(X,Y)$, is the graph whose vertex set is the set of all $n!$ configurations, and two configurations are adjacent if and only if they differ by a swap. 

The friends-and-strangers model is a unification and generalization of many problems, such as sliding blocks on graphs \cite{w74,y11}, and swapping tokens on graphs \cite{yama15,yama18}, hence this model received much attention. The recent study on friends-and-strangers problem focuses on the connectedness of the graph $\textsf{FS}(X,Y)$ from the perspectives of random graph theory and extremal graph theory \cite{a23,b22,j23}, and for special families of the graphs $X$ and $Y$ \cite{wy23}. So we define the following \textit{configuration-to-configuration} version of the decision problem related to the connectedness of friends-and-strangers graph, which will be denoted by \textsc{Fas}.

\begin{Definition}[The \textsc{Fas} Problem] Given two finite simple graphs $X$ and $Y$ of order $n$, and two configurations $\sigma$ and $\sigma'$, is there a path from $\sigma$ to $\sigma'$ in $\textsf{FS}(X,Y)$?
\end{Definition}

The main contribution of this paper is the following theorem on the computational complexity of \textsc{Fas}.

\begin{Theorem}\label{thm_main}
The \textsc{Fas} problem is $\textsf{PSPACE}$-complete, even if the location graph $X$ is a planar graph with maximum degree $3$.
\end{Theorem}

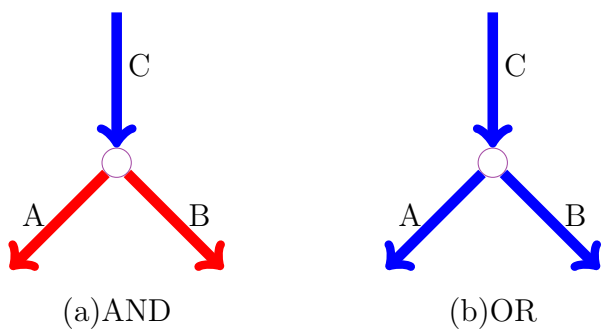
\begin{figure}[h]
\begin{center}
\begin{tikzpicture}
\node (v1) at ( 0,0) [circle,draw=violet!70] {};
\draw  [<-, line width=4, blue] (v1) -- (0,2) ;
\draw  [->, line width=4, red] (v1) -- (-1.414,-1.414) ; 
\draw  [->, line width=4, red] (v1) -- (1.414,-1.414) ; 
\node at (0,-2) {(a)AND};
\node at (0.3,1.3) {C};
\node at (-1.1,-0.7) {A};
\node at (1.1,-0.7) {B};

\node (v1) at ( 5,0) [circle,draw=violet!70] {};
\draw  [<-, line width=4, blue] (v1) -- (5,2) ;
\draw  [->, line width=4, blue] (v1) -- (5-1.414,-1.414) ;
\draw  [->, line width=4, blue] (v1) -- (6.414,-1.414) ;
\node at (5,-2) {(b)OR};
\node at (5+0.3,1.3) {C};
\node at (5-1.1,-0.7) {A};
\node at (5+1.1,-0.7) {B};

\end{tikzpicture}
\end{center}
\caption{\textsc{Ncl} (a) an AND vertex, (b) an OR vertex}\label{fig_ncl_vertex}
\end{figure}

We will prove Theorem \ref{thm_main} by reduction from the \textsc{Ncl} problem introduced by Hearn and Demaine in \cite{hd05}, which has been used in proving the $\textsf{PSPACE}$-completeness of several other problems \cite{bb12,c23,hly17}. \textsc{Ncl} is short for \textit{non-derministic constraint logic}, which is in fact a decision problem defined on $3$-regular directed graphs called \textsc{Ncl} graphs. Each edge of an \textsc{Ncl} graph has a weight of $1$ or $2$ (illustrated by red or blue edges respectively in Figure \ref{fig_ncl_vertex}). Moreover, there are only two kinds of vertices, the AND vertices and the OR vertices, in the \textsc{Ncl} graph as illustrated in Figure \ref{fig_ncl_vertex} up to the direction of their incident edges (i.e. the direction of the edges may vary). There is yet a constraint in the \textsc{Ncl} graph that the sum of the weights of incoming edges of each vertex must be at least $2$. With this constraint, the vertices kind of simulate the AND and OR logic gates. A configuration of the \textsc{Ncl} graph is an orientation of the edges satisfying the constraints. An edge can change its direction (called \textit{an edge flip}) if the configurations before and after the change both satisfy the constraints.

\begin{Definition}[\cite{hd05}, The \textit{configuration-to-configuration} \textsc{Ncl} Problem] Given two configurations $f$ and $g$ of an \textsc{Ncl} graph, can $g$ be obtained from $f$ by a sequence of edge flips (all intermediate configurations should satisfy the constraint)?
\end{Definition}

In \cite{hd05}, the following theorem is proved.

\begin{Theorem}[\cite{hd05}]
The \textsc{Ncl} problem is \textsf{PSPACE}-complete, even for planar \textsc{Ncl} graphs.
\end{Theorem}

\section{Proof of Theorem \ref{thm_main}}

\begin{proof}[Proof of Theorem \ref{thm_main}] We first show that \textsc{Fas} is in \textsf{PSPACE}. Given an instance of the \textsc{Fas} problem with graphs $X$ and $Y$ of order $n$, and two configurations $\sigma$ and $\sigma'$. If there is a path from $\sigma$ to $\sigma'$ in $\textsf{FS}(X,Y)$, then the path must be of length less than $n!$, which is the total number of different configurations. So an algorithm can try all possible walks starting from $\sigma$ non-deterministically. The algorithm either finds a walk to $\sigma'$, or returns a negative answer if all walks of length at most $n!$ have been exhausted. To keep track of the length of the walk, we need a counter of at most
$$\log_2 n!=\sum_{i=1}^n \log_2 i \leq n\log_2 n$$
binary bits. This implies that  \textsc{Fas}$\in$\textsf{NPSPACE}$=$\textsf{PSPACE}.


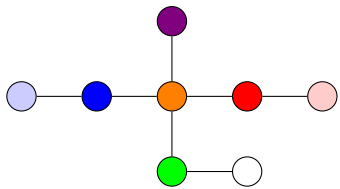
\begin{figure}[h]
\begin{center}
\begin{tikzpicture}[scale=0.5]

\node (lr) at (6,0) [circle,draw,fill=red!20] {};
\node (r) at (4,0) [circle,draw,fill=red] {};

\node (lb) at (-2,0) [circle,draw,fill=blue!20] {};
\node (b) at (0,0) [circle,draw,fill=blue] {};

\node (o) at (2,0) [circle,draw,fill=orange] {};
\node (violet) at (2,2) [circle,draw,fill=violet] {};

\node (g) at (2,-2) [circle,draw,fill=green] {};
\node (w) at (4,-2) [circle,draw] {};

\draw (lb) -- (b) -- (o)--(violet);
\draw (lr) -- (r) -- (o); \draw (w) -- (g) -- (o);

\end{tikzpicture}
\end{center}
\caption{Friends and strangers relations} \label{fig_fs_relation}
\end{figure}

Next, we show that \textsc{Fas} is \textsf{PSPACE}-hard by reduction from configuration-to-configuration \textsc{Ncl} problem. We will show how to construct a configuration of the \textsc{Fas} problem to simulate a configuration of the \textsc{Ncl} problem. To this end, we just need to construct an edge gadget (resp. a vertex gadget) in \textsc{Fas} for each edge (resp. vertex) in the \textsc{Ncl} graph. We have eight kinds of people in our construction, represented by eight colors: blue, light blue, red, light red, green, white, orange and violet. People of the same color are friends. Two people with different colors are friends if and only if there is an edge connecting these two colors in the graph shown in Figure \ref{fig_fs_relation}. Note that the graph $Y$ is a blown-up graph of this graph, and thus has a relatively high connectivity.

\begin{figure}[h]
\begin{center}
\begin{tikzpicture}[scale=0.5]

\draw  [<-, line width=6, blue] (6,4) -- (16,4) ;

\node (v0) at (4,2) [circle,draw,fill=green] {};
\node (v1) at (6,2) [circle,draw,fill=violet] {};
\node (v2) at (8,2) [circle,draw] {};
\node (v3) at (10,2) [circle,draw] {};
\node (v4) at (12,2) [circle,draw] {};
\node (v5) at (14,2) [circle,draw] {};
\node (v6) at (16,2) [circle,draw] {};
\node (v7) at (18,2) [circle,draw,fill=orange] {};

\node (v8) at (0,0) [circle,draw,fill=blue!20] {};
\node (v9) at (2,0) [circle,draw,fill=blue!20] {};
\node (v10) at (4,0) [circle,draw,fill=blue!20] {};
\node (v11) at (6,0) [circle,draw,fill=blue!20] {};
\node (v12) at (8,0) [circle,draw,fill=orange] {};

\node (v13) at (14,0) [circle,draw,fill=blue] {};
\node (v14) at (16,0) [circle,draw,fill=violet] {};
\node (v15) at (18,0) [circle,draw,fill=blue!20] {};
\node (v16) at (20,0) [circle,draw,fill=blue!20] {};
\node (v17) at (22,0) [circle,draw,fill=blue!20] {};

\draw (v1) -- (v0);
\draw (v2) -- (v1);
\draw (v3) -- (v2);
\draw (v4) -- (v3);
\draw (v5) -- (v4);
\draw (v6) -- (v5);
\draw (v7) -- (v6);

\draw (v11) -- (v1);

\draw (v9) -- (v8);
\draw (v10) -- (v9);
\draw (v11) -- (v10);
\draw (v12) -- (v11);

\draw (v14) -- (v6);

\draw (v14) -- (v13);
\draw (v15) -- (v14);
\draw (v16) -- (v15);
\draw (v17) -- (v16);

\node (r) at (24,0) {}; \node (l) at (-2,0) {};
\draw (v17)--(r); \draw (l)--(v8);

\node at (8.8,0) {$\alpha$}; \node at (18.8,2) {$\beta$}; \node at (13.2,0) {$\gamma$};

\end{tikzpicture}
\end{center}
\caption{Blue edge gadget (blue edge of NCL, with weight $2$)} \label{fig_edge_blue}
\end{figure}
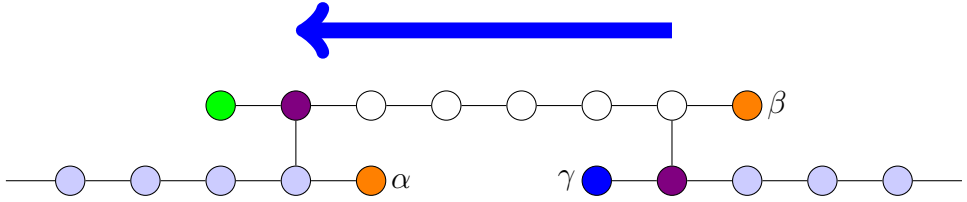

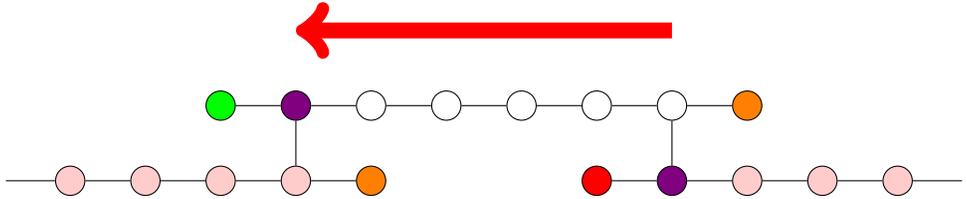
\begin{figure}[h]
\begin{center}
\begin{tikzpicture}[scale=0.5]

\draw  [<-, line width=6,  red] (6,4) -- (16,4) ;

\node (v0) at (4,2) [circle,draw,fill=green] {};
\node (v1) at (6,2) [circle,draw,fill=violet] {};
\node (v2) at (8,2) [circle,draw] {};
\node (v3) at (10,2) [circle,draw] {};
\node (v4) at (12,2) [circle,draw] {};
\node (v5) at (14,2) [circle,draw] {};
\node (v6) at (16,2) [circle,draw] {};
\node (v7) at (18,2) [circle,draw,fill=orange] {};

\node (v8) at (0,0) [circle,draw,fill=red!20] {};
\node (v9) at (2,0) [circle,draw,fill=red!20] {};
\node (v10) at (4,0) [circle,draw,fill=red!20] {};
\node (v11) at (6,0) [circle,draw,fill=red!20] {};
\node (v12) at (8,0) [circle,draw,fill=orange] {};

\node (v13) at (14,0) [circle,draw,fill=red] {};
\node (v14) at (16,0) [circle,draw,fill=violet] {};
\node (v15) at (18,0) [circle,draw,fill=red!20] {};
\node (v16) at (20,0) [circle,draw,fill=red!20] {};
\node (v17) at (22,0) [circle,draw,fill=red!20] {};

\draw (v1) -- (v0);
\draw (v2) -- (v1);
\draw (v3) -- (v2);
\draw (v4) -- (v3);
\draw (v5) -- (v4);
\draw (v6) -- (v5);
\draw (v7) -- (v6);

\draw (v11) -- (v1);

\draw (v9) -- (v8);
\draw (v10) -- (v9);
\draw (v11) -- (v10);
\draw (v12) -- (v11);

\draw (v14) -- (v6);

\draw (v14) -- (v13);
\draw (v15) -- (v14);
\draw (v16) -- (v15);
\draw (v17) -- (v16);

\node (r) at (24,0) {}; \node (l) at (-2,0) {};
\draw (v17)--(r); \draw (l)--(v8);

\end{tikzpicture}
\end{center}
\caption{Red edge gadget (red edge of NCL, with weight $1$)} \label{fig_edge_red}
\end{figure}

A blue edge gadget is illustrated in Figure \ref{fig_edge_blue}, which is a fragment of the location graph $X$ with its mapping to $Y$. The two open-ended edges on the left and on the right are to be connected to the vertex gadgets that will be introduced later. Note that at its current state as shown in Figure \ref{fig_edge_blue}, no people can swap with its neighbors, except those people of the same color (swapping two people of the same color does not make any difference at all). But if a blue people is coming from outside through the open-ended edge on the left (in fact the blue is coming from a vertex gadget as we will see soon), he can go all the way to the right to the location $\alpha$, by swapping with the light blue and the orange people. The orange then swaps location with the violet, and sets the green free, because the green is a friend of orange but not violet. Note that the incoming blue is now locked at the location $\alpha$. All the green can do is to go all the way to the right to the location $\beta$, by a sequence of swapping. The orange originally at the location $\beta$ is now a neighbor of the violet, so it swaps location with the violet. Now the blue at location $\gamma$ is set free, and can leave the edge gadget by going to the right. In short, the blue edge gadget can lock up a blue on one side, and set a blue free on the other side, it simulates the direction of an edge of \textsc{Ncl} graphs, pointing to the side that a blue has been released. A red edge gadget is almost the same as the blue edge gadget, except that the blue and light blue people are replaced by red and light red, respectively (see Figure \ref{fig_edge_red}).


\begin{figure}[h]
\begin{center}
\begin{tikzpicture}[scale=0.5]

\node (v0) at (4,2) [circle,draw,fill=green] {};
\node (v1) at (6,2) [circle,draw,fill=violet] {};
\node (v2) at (8,2) [circle,draw] {};
\node (v3) at (10,2) [circle,draw] {};
\node (v4) at (12,2) [circle,draw] {};
\node (v5) at (14,2) [circle,draw] {};
\node (v6) at (16,2) [circle,draw] {};
\node (v7) at (18,2) [circle,draw,fill=orange] {};

\node (v8) at (0,0) [circle,draw,fill=blue!20] {};
\node (v9) at (2,0) [circle,draw,fill=blue!20] {};
\node (v10) at (4,0) [circle,draw,fill=blue!20] {};
\node (v11) at (6,0) [circle,draw,fill=blue!20] {};
\node (v12) at (8,0) [circle,draw,fill=orange] {};

\node (v13) at (14,0) [circle,draw,fill=blue] {};
\node (v14) at (16,0) [circle,draw,fill=violet] {};
\node (v15) at (18,0) [circle,draw,fill=blue!20] {};
\node (v16) at (20,0) [circle,draw,fill=blue!20] {};
\node (v17) at (22,0) [circle,draw,fill=blue!20] {};

\draw (v1) -- (v0);
\draw (v2) -- (v1);
\draw (v3) -- (v2);
\draw (v4) -- (v3);
\draw (v5) -- (v4);
\draw (v6) -- (v5);
\draw (v7) -- (v6);

\draw (v11) -- (v1);

\draw (v9) -- (v8);
\draw (v10) -- (v9);
\draw (v11) -- (v10);
\draw (v12) -- (v11);

\draw (v14) -- (v6);

\draw (v14) -- (v13);
\draw (v15) -- (v14);
\draw (v16) -- (v15);
\draw (v17) -- (v16);

\node (w0) at (4,6) [circle,draw,fill=green] {};
\node (w1) at (6,6) [circle,draw,fill=violet] {};
\node (w2) at (8,6) [circle,draw] {};
\node (w3) at (10,6) [circle,draw] {};
\node (w4) at (12,6) [circle,draw] {};
\node (w5) at (14,6) [circle,draw] {};
\node (w6) at (16,6) [circle,draw] {};
\node (w7) at (18,6) [circle,draw,fill=orange] {};

\node (w8) at (0,4) [circle,draw,fill=blue!20] {};
\node (w9) at (2,4) [circle,draw,fill=blue!20] {};
\node (w10) at (4,4) [circle,draw,fill=blue!20] {};
\node (w11) at (6,4) [circle,draw,fill=blue!20] {};
\node (w12) at (8,4) [circle,draw,fill=orange] {};

\node (w13) at (14,4) [circle,draw,fill=blue] {};
\node (w14) at (16,4) [circle,draw,fill=violet] {};
\node (w15) at (18,4) [circle,draw,fill=blue!20] {};
\node (w16) at (20,4) [circle,draw,fill=blue!20] {};
\node (w17) at (22,4) [circle,draw,fill=blue!20] {};

\draw (w1) -- (w0);
\draw (w2) -- (w1);
\draw (w3) -- (w2);
\draw (w4) -- (w3);
\draw (w5) -- (w4);
\draw (w6) -- (w5);
\draw (w7) -- (w6);

\draw (w11) -- (w1);

\draw (w9) -- (w8);
\draw (w10) -- (w9);
\draw (w11) -- (w10);
\draw (w12) -- (w11);

\draw (w14) -- (w6);

\draw (w14) -- (w13);
\draw (w15) -- (w14);
\draw (w16) -- (w15);
\draw (w17) -- (w16);

\node (u0) at (4,10) [circle,draw,fill=green] {};
\node (u1) at (6,10) [circle,draw,fill=violet] {};
\node (u2) at (8,10) [circle,draw] {};
\node (u3) at (10,10) [circle,draw] {};
\node (u4) at (12,10) [circle,draw] {};
\node (u5) at (14,10) [circle,draw] {};
\node (u6) at (16,10) [circle,draw] {};
\node (u7) at (18,10) [circle,draw,fill=orange] {};

\node (u8) at (0,8) [circle,draw,fill=blue!20] {};
\node (u9) at (2,8) [circle,draw,fill=blue!20] {};
\node (u10) at (4,8) [circle,draw,fill=blue!20] {};
\node (u11) at (6,8) [circle,draw,fill=blue!20] {};
\node (u12) at (8,8) [circle,draw,fill=orange] {};

\node (u13) at (14,8) [circle,draw,fill=blue] {};
\node (u14) at (16,8) [circle,draw,fill=violet] {};
\node (u15) at (18,8) [circle,draw,fill=blue!20] {};
\node (u16) at (20,8) [circle,draw,fill=blue!20] {};
\node (u17) at (22,8) [circle,draw,fill=blue!20] {};

\draw (u1) -- (u0);
\draw (u2) -- (u1);
\draw (u3) -- (u2);
\draw (u4) -- (u3);
\draw (u5) -- (u4);
\draw (u6) -- (u5);
\draw (u7) -- (u6);

\draw (u11) -- (u1);

\draw (u9) -- (u8);
\draw (u10) -- (u9);
\draw (u11) -- (u10);
\draw (u12) -- (u11);

\draw (u14) -- (u6);

\draw (u14) -- (u13);
\draw (u15) -- (u14);
\draw (u16) -- (u15);
\draw (u17) -- (u16);

\node (o1) at (-2,2) [circle,draw,fill=blue] {};
\node (o2) at (-2,6) [circle,draw,fill=blue] {};
\node (m) at (-2,4) [circle,draw,fill=blue!20] {};
\node (m1) at (-2,0) [circle,draw,fill=blue!20] {};
\node (m2) at (-2,8) [circle,draw,fill=blue!20] {};

\draw (m2) -- (u8);\draw (m2) -- (o2);
\draw (m1) -- (v8); \draw (m1) -- (o1);
\draw (m) -- (w8);
\draw (m) -- (o1); \draw (m) -- (o2);

\draw [dashed, ultra thick] (-3,-1)--(-1,-1)--(-1,9)--(-3,9)--(-3,-1);

\end{tikzpicture}
\end{center}
\caption{The OR vertex gadget} \label{fig_or}
\end{figure}
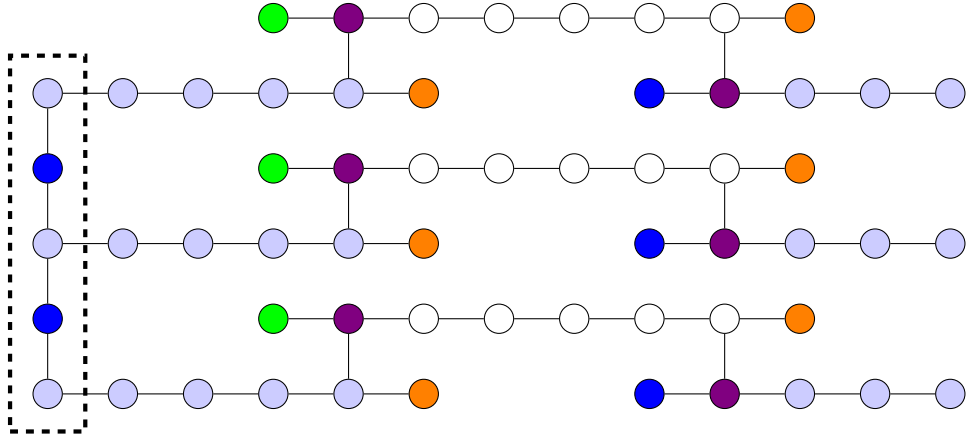

The OR vertex gadget is illustrated by the fragment enclosed by the dashed rectangle in Figure \ref{fig_or}, and is connected to three blue edge gadgets. All three blue edge gadgets are pointing to the OR vertex gadget at the moment. But there are only two free blue inside the OR vertex gadget, so at any moment, at most two edges can change their orientation from pointing inwards to pointing outwards. This simulates the logic of OR vertex in \textsc{Ncl} graphs.


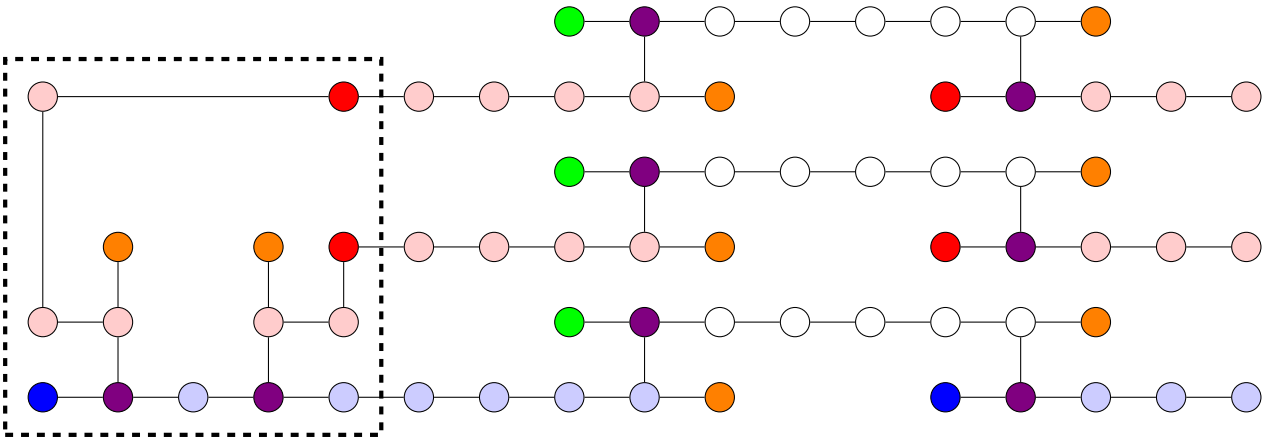
\begin{figure}[h]
\begin{center}
\begin{tikzpicture}[scale=0.5]

\node (v0) at (4,2) [circle,draw,fill=green] {};
\node (v1) at (6,2) [circle,draw,fill=violet] {};
\node (v2) at (8,2) [circle,draw] {};
\node (v3) at (10,2) [circle,draw] {};
\node (v4) at (12,2) [circle,draw] {};
\node (v5) at (14,2) [circle,draw] {};
\node (v6) at (16,2) [circle,draw] {};
\node (v7) at (18,2) [circle,draw,fill=orange] {};

\node (v8) at (0,0) [circle,draw,fill=blue!20] {};
\node (v9) at (2,0) [circle,draw,fill=blue!20] {};
\node (v10) at (4,0) [circle,draw,fill=blue!20] {};
\node (v11) at (6,0) [circle,draw,fill=blue!20] {};
\node (v12) at (8,0) [circle,draw,fill=orange] {};

\node (v13) at (14,0) [circle,draw,fill=blue] {};
\node (v14) at (16,0) [circle,draw,fill=violet] {};
\node (v15) at (18,0) [circle,draw,fill=blue!20] {};
\node (v16) at (20,0) [circle,draw,fill=blue!20] {};
\node (v17) at (22,0) [circle,draw,fill=blue!20] {};

\draw (v1) -- (v0);
\draw (v2) -- (v1);
\draw (v3) -- (v2);
\draw (v4) -- (v3);
\draw (v5) -- (v4);
\draw (v6) -- (v5);
\draw (v7) -- (v6);

\draw (v11) -- (v1);

\draw (v9) -- (v8);
\draw (v10) -- (v9);
\draw (v11) -- (v10);
\draw (v12) -- (v11);

\draw (v14) -- (v6);

\draw (v14) -- (v13);
\draw (v15) -- (v14);
\draw (v16) -- (v15);
\draw (v17) -- (v16);

\node (w0) at (4,6) [circle,draw,fill=green] {};
\node (w1) at (6,6) [circle,draw,fill=violet] {};
\node (w2) at (8,6) [circle,draw] {};
\node (w3) at (10,6) [circle,draw] {};
\node (w4) at (12,6) [circle,draw] {};
\node (w5) at (14,6) [circle,draw] {};
\node (w6) at (16,6) [circle,draw] {};
\node (w7) at (18,6) [circle,draw,fill=orange] {};

\node (w8) at (0,4) [circle,draw,fill=red!20] {};
\node (w9) at (2,4) [circle,draw,fill=red!20] {};
\node (w10) at (4,4) [circle,draw,fill=red!20] {};
\node (w11) at (6,4) [circle,draw,fill=red!20] {};
\node (w12) at (8,4) [circle,draw,fill=orange] {};

\node (w13) at (14,4) [circle,draw,fill=red] {};
\node (w14) at (16,4) [circle,draw,fill=violet] {};
\node (w15) at (18,4) [circle,draw,fill=red!20] {};
\node (w16) at (20,4) [circle,draw,fill=red!20] {};
\node (w17) at (22,4) [circle,draw,fill=red!20] {};

\draw (w1) -- (w0);
\draw (w2) -- (w1);
\draw (w3) -- (w2);
\draw (w4) -- (w3);
\draw (w5) -- (w4);
\draw (w6) -- (w5);
\draw (w7) -- (w6);

\draw (w11) -- (w1);

\draw (w9) -- (w8);
\draw (w10) -- (w9);
\draw (w11) -- (w10);
\draw (w12) -- (w11);

\draw (w14) -- (w6);

\draw (w14) -- (w13);
\draw (w15) -- (w14);
\draw (w16) -- (w15);
\draw (w17) -- (w16);

\node (u0) at (4,10) [circle,draw,fill=green] {};
\node (u1) at (6,10) [circle,draw,fill=violet] {};
\node (u2) at (8,10) [circle,draw] {};
\node (u3) at (10,10) [circle,draw] {};
\node (u4) at (12,10) [circle,draw] {};
\node (u5) at (14,10) [circle,draw] {};
\node (u6) at (16,10) [circle,draw] {};
\node (u7) at (18,10) [circle,draw,fill=orange] {};

\node (u8) at (0,8) [circle,draw,fill=red!20] {};
\node (u9) at (2,8) [circle,draw,fill=red!20] {};
\node (u10) at (4,8) [circle,draw,fill=red!20] {};
\node (u11) at (6,8) [circle,draw,fill=red!20] {};
\node (u12) at (8,8) [circle,draw,fill=orange] {};

\node (u13) at (14,8) [circle,draw,fill=red] {};
\node (u14) at (16,8) [circle,draw,fill=violet] {};
\node (u15) at (18,8) [circle,draw,fill=red!20] {};
\node (u16) at (20,8) [circle,draw,fill=red!20] {};
\node (u17) at (22,8) [circle,draw,fill=red!20] {};

\draw (u1) -- (u0);
\draw (u2) -- (u1);
\draw (u3) -- (u2);
\draw (u4) -- (u3);
\draw (u5) -- (u4);
\draw (u6) -- (u5);
\draw (u7) -- (u6);

\draw (u11) -- (u1);

\draw (u9) -- (u8);
\draw (u10) -- (u9);
\draw (u11) -- (u10);
\draw (u12) -- (u11);

\draw (u14) -- (u6);

\draw (u14) -- (u13);
\draw (u15) -- (u14);
\draw (u16) -- (u15);
\draw (u17) -- (u16);

\node (o1) at (-4,0) [circle,draw,fill=violet] {};
\node (o2) at (-8,0) [circle,draw,fill=violet] {};
\node (p1) at (-4,2) [circle,draw,fill=red!20] {};
\node (p2) at (-8,2) [circle,draw,fill=red!20] {};
\node (q1) at (-4,4) [circle,draw,fill=orange] {};
\node (q2) at (-8,4) [circle,draw,fill=orange] {};

\node (o3) at (-10,0) [circle,draw,fill=blue] {};

\node (m) at (-2,0) [circle,draw,fill=blue!20] {};
\node (m2) at (-6,0) [circle,draw,fill=blue!20] {};
\draw (m) -- (v8);\draw (m2) -- (o1);\draw (m2) -- (o2);

\node (n) at (-2,2) [circle,draw,fill=red!20] {};
\node (n1) at (-2,4) [circle,draw,fill=red] {};
\node (n2) at (-2,8) [circle,draw,fill=red] {};
\node (n3) at (-10,8) [circle,draw,fill=red!20] {};
\node (n4) at (-10,2) [circle,draw,fill=red!20] {};

\draw (m) -- (o1);

\draw (o2) -- (o3);

\draw (q1) -- (p1);
\draw (q2) -- (p2);
\draw (o1) -- (p1);
\draw (o2) -- (p2);

\draw (n2) -- (n3);\draw (n3) -- (n4); \draw (n4) -- (p2);
\draw (n2) -- (u8);
\draw (n1) -- (w8);\draw (n) -- (n1);\draw (n) -- (p1);

\draw [dashed, ultra thick] (-11,-1)--(-1,-1)--(-1,9)--(-11,9)--(-11,-1);

\end{tikzpicture}
\end{center}
\caption{The AND vertex gadget} \label{fig_and}
\end{figure}

The AND vertex gadget is the fragment enclosed by the dashed rectangle in Figure \ref{fig_and}, and is connecting to two incoming red edge gadgets and one incoming blue edge gadget. It can be checked that at any time, the AND vertex gadget can either lock up the blue inside the gadget and let go two red, which in turn allow two red edge gadgets pointing outwards; or lock up two red inside the gadget and let the blue go, which allow the blue edge gadgets pointing outwards. This simulates the logic of AND vertex in \textsc{Ncl} graphs.

Given an instance of planar \textsc{Ncl} graph with two configurations $f$ and $g$, we can construct the two corresponding configurations $\sigma$ and $\sigma'$ of an instance of the \textsc{Fas} problem by putting together the edge gadgets and vertex gadgets mentioned above. By the construction, $\sigma$ and $\sigma'$ are connected in the graph \textsf{FS}$(X,Y)$ if and only if $f$ can be changed to $g$ by edge flips. This completes the reduction. Finally, note that the location graph $X$ (i.e. all the edge gadgets and vertex gadgets) is planar with maximum degree $3$, so the proof is complete.   \end{proof}

\noindent\textbf{Remark.} It is also natural to consider a variant of \textsc{Fas} problem: given a configuration $\sigma$, a particular people $p$ and a target location $t$, is it possible for $p$ to reach location $t$ starting from its current location in $\sigma$? By reduction from the \textit{configuration-to-edge} variant of \textsc{Ncl} problem \cite{hd05}, our method can also prove the \textsf{PSPACE}-completeness of this variant of \textsc{Fas} problem.


\section*{Acknowledgements}
The first author was supported by the Research Fund of Guangdong University of Foreign Studies (Nos. 297-ZW200011 and 297-ZW230018), and the National Natural Science Foundation of China (No. 61976104).


\end{document}